\documentclass[12pt]{amsart}
\usepackage{amsmath,amssymb,amsthm,color, euscript, enumerate,comment}
\usepackage{dsfont}
\usepackage{longtable}

\renewcommand{\l}{\left}
\renewcommand{\r}{\right}

\newcommand{\maru}[1]{{\ooalign{\hfil#1\/\hfil\crcr
\raise.167ex\hbox{\mathhexbox20D}}}}

\newcommand{\ruby}[2]{%
 \leavevmode
 \setbox0=\hbox{#1}%
 \setbox1=\hbox{\tiny #2}%
 \ifdim\wd0>\wd1 \dimen0=\wd0 \end{lemma}se \dimen0=\wd1 \fi
 \hbox{%
   \kanjiskip=0pt plus 2fil
   \xkanjiskip=0pt plus 2fil
   \vbox{%
     \hbox to \dimen0{%
       \tiny \hfil#2\hfil}%
     \nointerlineskip
     \hbox to \dimen0{\mathstrut\hfil#1\hfil}}}}

\newcommand{\la}{\langle}
\newcommand{\ra}{\rangle}

\newcommand{\Z}{\mathbb{Z}}
\newcommand{\C}{\mathbb{C}}

\newcommand{\Q}{\mathbb{Q}}

\newcommand{\Sym}{{\rm Sym}}

\newcommand{\aut}{\mathrm{Aut}\,}
\newcommand{\Aut}{\mathrm{Aut}\,}

\renewcommand{\hom}{\mathrm{Hom}}

\newcommand{\be}{\beta}
\newcommand{\al}{\alpha}

\newcommand{\Span}{\mathrm{Span}}
\newcommand{\Stab}{\mathrm{Stab}}

\makeatletter \@addtoreset{equation}{section}

\theoremstyle{plain}
\newtheorem{theorem}{Theorem}[section]
\newtheorem{proposition}[theorem]{Proposition}
\newtheorem{lemma}[theorem]{Lemma}
\newtheorem{corollary}[theorem]{Corollary}

\theoremstyle{definition}
\newtheorem{definition}[theorem]{Definition}

\theoremstyle{remark}
\newtheorem{remark}[theorem]{Remark}
\numberwithin{equation}{section}

\title[Automorphism groups of cyclic orbifolds]{Completely fixed point free isometry and cyclic orbifold of lattice vertex operator algebras}

\author{Hsian-Yang Chen}
  \address[H.Y.  Chen]{ National University of Tainan, Tainan  70005, Taiwan}
\email{hychen@mail.nutn.edu.tw} 
 \thanks{H.Y. Chen is supported by   NSTC grant  111-2115-M-024-001-MY2  of Taiwan}
\author{Ching Hung Lam} %
\address[C. H. Lam] {Institute of Mathematics, Academia Sinica, Taipei 10617, Taiwan} 
\email{chlam@math.sinica.edu.tw} 
\date{}
\thanks{C.H. Lam is supported by   NSTC grant  110-2115-M-001-003-MY3   of Taiwan}

\newcommand{\sfr}[2]{\leavevmode\kern-.1em
  \raise.5ex\hbox{\the\scriptfont0 #1}\kern-.1em
  /\kern-.15em\lower.25ex\hbox{\the\scriptfont0 #2}}

\begin{document}

\begin{abstract}
We continue our study of  cyclic orbifolds of lattice vertex operator algebras and their full automorphism groups. We consider  some special isometry $g\in O(L)$ such that  $g^i$ is fixed point free on $L$ for any $1\leq i\leq |g|-1$. 
We show that when $L_2=\emptyset$ and $g^i$ is fixed point free on $L$ for any $1\leq i\leq |g|-1$, $V_L^{\hat{g}}$ has extra automorphisms implies either (1) the order of $g$ is   a prime or  (2) $L$ is isometric  to the Leech lattice or some coinvariant sublattices of the Leech lattice. 
\end{abstract}

\maketitle


\section{Introduction}
We continue our study of  cyclic orbifolds of lattice vertex operator algebras and their full automorphism groups.
Let $V$ be a vertex operator algebra (abbreviated  as VOA) and let $g$ be  an automorphism of $V$ of finite order $n$. The fixed-point subspace $$V^g = \{ v \in V\mid  g v = v \}$$ is 
a subVOA and is often called an orbifold subVOA.    
When $V=V_L$ is a lattice VOA and $g$ is a lift of the $(-1)$ isometry of $L$, the full automorphism group of $V_L^g =V_L^+$ has been determined in \cite{Sh04,Sh06}.  The case when $|g|=p$ is a prime and $L$ is rootless has also been studied in \cite{LS21}. It is shown that a cyclic orbifold $V_L^{\hat{g}}$ contains extra automorphisms if and only if the rootless even lattice $L$ can be constructed by Construction B from a code over $\Z_p$ or is isometric to the coinvariant lattice of the Leech lattice associated with some isometry of order $p$. The key method  is to study the orbit
of the irreducible module $V_{L}(1) =\{v\in V_L\mid gv=e^{2\pi i/p} v\}$ under the conjugate actions of $\Aut(V_{L}^{\hat{g}})$. Using a similar method, certain orbifold vertex operator algebras  associated with some coinvariant lattices of the Leech lattice are studied in \cite{BLS,Lam1,Lam2}. In addition,  a certain criterion for the existence of extra automorphisms was also discussed in \cite{Lam1,LS21}. In this article, we consider another special case. We assume that  $g^i$ is fixed point free on $L$ for any $1\leq i\leq |g|-1$. We call such a $g\in O(L)$  a completely fixed point free isometry of $L$. 
As our main result, we show that when $L_2=\emptyset$ and $g\in O(L)$ is completely fixed point free, $V_L^{\hat{g}}$ has extra automorphisms implies that  either (1) the order of $g$ is   a prime or  (2) $L$ is isometric  to the Leech lattice or some coinvariant sublattices of the Leech lattice. We should note that a special case when $|g|=4$ is also studied in \cite{CL21}.

\section{Preliminaries}

\subsection{Lattices}\label{S:lattice}

Let $L$ be an even lattice with the positive-definite bilinear form $(\cdot | \cdot )$. We 
denote its isometry group by $O(L)$. 
The discriminant group $\mathcal{D}(L)$ is defined to 
be the quotient group $L^*/L$, where  $L^*=\{ \alpha\in \Q\otimes_\Z L\mid ( \alpha| L)\subset \Z \}$  is the dual lattice of $L$. 

Let $g\in O(L)$ of order $n$.
The fixed-point sublattice $L^g$ of $g$ and the \emph{coinvariant lattice} $L_g$ of $g$ are defined to be 
\begin{equation}
	L^g=\{\alpha\in L\mid g\alpha=\alpha\}\quad \text{ and } \quad L_g = \{ \alpha\in L \mid (\alpha| \beta) =0 \text{ for all } \beta\in L^g\}.\label{Eq:coinv}
\end{equation}
Clearly,  $$\mathrm{rank}(L)=\mathrm{rank}( L^g)+\mathrm{rank} (L_g).$$
Notice that the restriction of $g$ to $L_g$ is a fixed-point free isometry of order $n$. 

\subsection{Lattice VOAs and their automorphism groups}
We recall the structure of the automorphism group of a lattice VOA $V_L$.   
Let $L$ be an even lattice with the (positive-definite) bilinear form $( \ | \ )$.  
Consider a central extension 
$\hat{L}=\{\pm e^\al\mid \al\in L\}$ 
of $L$ by $\pm 1$
such that $e^{\al} e^\be = (-1)^{(\al |\be)} e^{\be}e^\al$.

Let $O(\hat{L})=\{g\in\aut(\hat{L})\mid \bar{g}\in O(L)\}$.
For $g\in O(\hat{L})$, let $\bar{g}$ be the map $L\to L$ defined by $g(e^\alpha)\in\{\pm e^{\bar{g}(\alpha)}\}$.
By \cite[Proposition 5.4.1]{FLM},  we have an exact sequence
\[
  1 \to \hom (L,\Z/2\Z) \to  O(\hat{L})\to  O(L) \to  1.
\]
Let
$$
  N(V_L) = \l\la \exp(a_{(0)}) \mid a\in (V_L)_1 \r\ra
$$
be the normal subgroup of $\Aut(V_L)$ generated by the inner automorphisms $\exp(a_{(0)})$.
Recall that $O(\hat{L})$ can be viewed a subgroup of $\aut(V_L)$ (cf.~loc.~cit.).


\begin{theorem}[\cite{DN}]\label{aut}
Let $L$ be a positive definite even lattice.
Then
\[
  \aut (V_L) = N(V_L)\,O(\hat{L})
\]
Moreover, the intersection $N(V_L)\cap O(\hat{L})$ contains a subgroup
$\hom (L,\Z/2\Z)$ and the quotient $\aut (V_L)/N(V_L)$ is isomorphic
to a quotient group of $O(L)$.
\end{theorem}

\begin{remark}\label{L2=0}
 When $L(2)=\emptyset$,  $(V_L)_1=\Span\{ \alpha{(-1)} \cdot 1 \mid \alpha\in L\}$ and   the normal subgroup
  $N(V_L)= \{\exp(\lambda \alpha{(0)}) \mid \alpha\in L,~\lambda \in \C\}$ is abelian. In this case, $N(V_L) \cap O(\hat{L})= \hom (L,\Z/2\Z)$ and $\aut (V_L)/N(V_L) \cong O(L)$. In particular, we have an exact sequence
  \begin{equation}\label{eq:5.3}
    1\to N(V_L) \to \aut(V_L) \stackrel{~\varphi~}\to O(L)\to 1.
  \end{equation}
\end{remark}

The following theorem can be proved by the same argument as in \cite[Theorem 
5.15]{LY2} (See also \cite[Theorem 2.7]{LY}).
 
\begin{theorem}
 \label{normalizer}
Let $L$ be a positive-definite rootless even lattice.
Let $g$ be a fixed-point free isometry of $L$ of finite order and
$\hat{g}$ a lift of $g$ in $O(\hat{L})$.
Then we have the following exact sequences:
\[
\begin{split}
 1\longrightarrow \hom(L/(1-g)L, \C^*)
  \longrightarrow N_{\aut(V_L)}(\langle \hat{g}\rangle)
  \stackrel{\varphi}\longrightarrow N_{O(L)}(\langle g\rangle) \longrightarrow 1;\\
   1\longrightarrow \hom(L/(1-g)L, \C^*)
  \longrightarrow C_{\aut(V_L)}(\hat{g})
  \stackrel{\varphi}\longrightarrow C_{O(L)}(g) \longrightarrow 1. 
\end{split}
\]
\end{theorem}

Let $h\in N_{\aut(V_L)}(\langle \hat{g}\rangle )$. Then it is clear that $hx \in V_L^{\hat{g}}$ for any $x\in 
V_L^{\hat{g}}$. Therefore, $N_{\aut(V_L)}(\langle \hat{g}\rangle)$ acts on $V_L^{\hat{g}}$ and there is a group 
homomorphism $$f: N_{\aut(V_L)}(\langle \hat{g}\rangle)/ \langle \hat{g}\rangle \longrightarrow 
\aut(V_L^{\hat{g}}).$$ The key question is to determine the image of $f$ and if $f$ is surjective. In 
particular, one would like to determine if there exist automorphisms in $\aut(V_L^{\hat{g}})$ 
which are not induced from $N_{\aut(V_L)}(\langle \hat{g}\rangle)$. We call such an 
automorphism an \emph{extra automorphism}.

\section{Irreducible modules for cyclic orbifolds of lattice VOAs}

Next we review a few facts about the irreducible modules for the orbifold VOA $V_L^{\hat{g}}$.  Recall from \cite{Mi,CM} that the orbifold VOA $V_{L}^{\hat{g}}$ is  $C_2$-cofinite and rational. Moreover, any irreducible $V_{L}^{\hat{g}}$-module is a submodule of an irreducible $\hat{g}^i$-twisted $V_{L}$-module for some $0\le i\le |\hat{g}|-1$  \cite{DRX}.

\begin{definition}\label{Mconj}
	Let $V$ be a VOA and $\tau\in\Aut(V)$. Let $g\in\Aut(V)$ and let $M=(M,Y_M)$ be a $g$-twisted $V$-module.
	The \emph{$\tau$-conjugate} $(M\circ \tau, Y_{M\circ \tau} (\cdot, z))$ of $M$ is defined as follows:
	\begin{equation}
		\begin{split}
			& M\circ \tau =M \quad \text{ as a vector space;}\\
			& Y_{M\circ \tau} (a, z) = Y_M(\tau a, z)\quad \text{ for any } a\in V.
		\end{split}\label{Eq:conjact}
	\end{equation}
	Then $(M\circ \tau, Y_{M\circ \tau} (\cdot, z))$ defines a $\tau g\tau^{-1}$-twisted  $V$-module.
\end{definition}

An irreducible (untwisted or twisted) module $M$ of $V_L$ is said to be $\hat{g}$-stable if 
$M\circ \hat{g}\cong M$. In this case, $\hat{g}$ acts on $M$ and we denote  the eigenspaces of $\hat{g} $ on $M$ by 
\[
M(j)= \{ x\in M\mid  \hat{g} x= e^{2\pi \sqrt{-1}j / n} x\}, \quad  0 \leq j\leq n-1,\ n=|\hat{g}|.    
\]

If $V_{\lambda+L}$ is $\hat{g}$-stable, or equivalently, $(1-g)\lambda\in L$, then $V_{\lambda+L}(j)$ is a simple current module of $V_L^{\hat{g}}$ (see for example \cite{Lam19}).  It is also known that the number of inequivalent  irreducible $\hat{g}^i$-twisted modules is equal to the number of inequivalent irreducible $\hat{g}^i$-stable modules of $V_L$ and all irreducible $\hat{g}^i$-twisted modules are $\hat{g}$-stable \cite{DLM2}. 
 
Let $P_0^{g^i}:L^* \to \Q\otimes_\Z L^{g^i}$ be the orthogonal projection and use $(L^*/L)^{g^i}$ to denote the set of cosets of $L$ in $L^*$ fixed by $g^i$.  
Then  $V_L$ has exactly $|(L^*/L)^{g^i}|$ irreducible $\hat{g}^i$-twisted $V_L$-modules, up to isomorphism. The irreducible $\hat{g}^i$-twisted $V_L$-modules have been constructed in \cite{Le,DL} explicitly and are classified in \cite{BK04}.
They are given by 
\begin{equation}
V_{\lambda+L}[\hat{g}^i]= M(1)[g^i]\otimes\C[P_0^{g^i}(\lambda+L)]\otimes T_{\tilde{\lambda}}, 
\qquad \text{ for } \lambda+L\in (L^*/L)^{g^i},   \label{twmodule0}
\end{equation}
where $M(1)[g^i]$ is the ``$g^i$-twisted" free bosonic space, $\C[\lambda+P_0^{g^i}(L)]$ is a module for the group algebra of $P_0^{g^i}(L)$ and $T_{\tilde{\lambda}}$ is an irreducible module for a certain ``$g^i$-twisted" central extension of $L_{g^i}$ associated with $\lambda$ (see \cite[Propositions 6.1 and 6.2]{Le} and \cite[Remark 4.2]{DL}).

\section{Orbifold VOAs having Extra automorphisms}
From now on, we study the orbifold lattice VOA $V_L^{\hat{g}}$ associated with an isometry $g\in O(L)$.  One of the main purposes is to try to classify all pairs $(L,g)$ such that $V_L^{\hat{g}}$ has extra automorphisms. In this article, we consider a special case such that $g^i$ is fixed point free for any $1\leq i\leq |g|-1$ on $L$.

\subsection{Completely fixed point free isometries}

\begin{definition}
	An isometry $g\in O(L)$ is said to be completely fixed point free if  $g^i$ is also fixed point free on $L$ for any $1\leq i\leq |g|-1$.
\end{definition}

If $g\in O(L)$ is completely fixed point free, then only the primitive $n$-th roots of unity will appear as the eigenvalues of $g$ and the minimal polynomial of $g$ on $L$ is given by the $n$-th cyclotomic polynomial $\Phi_n(x)$, where 
$n=|g|$. In this case, we have 
\begin{equation}
	\dim\mathfrak{h}_{(i)}=
	\begin{cases}
		\mathrm{rank}(L)/\varphi(n), &  \text{ if } (n,i)=1,\\
		0, & \text{ otherwise,} 
	\end{cases}
\end{equation}
where $\mathfrak{h}= \C\otimes_\Z L$, $\mathfrak{h}_{(j)}= \mathfrak{h}_{(j; g)}=\{v\in \mathfrak{h}\mid gv= e^{2\pi i\frac{j}{n} } v\}$ and 
$\varphi$ is the Euler totient function, i.e., $\varphi(n)$ is the number of positive integers $1\leq i\leq n-1$ which are relatively prime to $n$.

The following facts are well-known and will be used in the next section for classifying the possible pairs for $L$ and $g$.
 
\begin{lemma}\label{phi(1)}
	Let $$\Phi_n(x)=\prod_{1 \leq k \leq n,\atop (n,k)=1} (x-e^{2\pi i k/n})$$ be the $n$-th cyclotomic polynomial.  Then 
	\[
	\Phi_n(1) =
	\begin{cases}
		1 & \text{if $n$ is not a prime power},\\
		p & \text{if $n=p^k$ is a prime power}.
	\end{cases}
	\]
\end{lemma}

\begin{lemma}
	Suppose $n=p^k$ is a prime power. Then 
	\[
	\Phi_n(x)= \Phi_p(x^{p^{k-1}})= \sum_{i=1}^{p-1} x^{i p^{k-1}}.
	\] 
\end{lemma}

\subsection{Characterization  of $L$ and $g$} 

Next  we will try to classify all even lattice $L$ with $L_2=\emptyset$ and completely fixed point free $g\in O(L)$ so that $V_L^{\hat{g}}$ has an extra automorphism. 

From now on, we always assume that $g$ is completely fixed point free and $L_2=\emptyset$.  We also assume that $V_L^{\hat{g}}$ has an extra automorphism $\sigma$.  

By \cite[Theorem 2.1]{Sh07}, the extra automorphism $\sigma$ does not preserve the set 
\[
\{ V_L(r)\circ \sigma \mid 0\leq r\leq |g|-1\}.  
\]
under the conjugate action.  In other words, $V_L(1)\circ \sigma$ is isomorphic to a simple current module of $V_L^{\hat{g}}$ 
not containing in $\{ V_L(r) \mid 1\leq r\leq |g|-1 \}$. 
By the classification of simple current modules of $V_L^{\hat{g}}$, $V_L(1)\circ \sigma$ is either isomorphic to 
\begin{enumerate}[(I)]
\item  $V_{\lambda+L}(r)$ for some $\lambda\in L^*\setminus L$ with $(1-g)\lambda\in L$ and $0\leq r\leq n-1$; or 
 
\item an irreducible $V_L^{\hat{g}}$-submodule $V_{\lambda+L}^T[\hat{g}^s](j)$ for some $1\leq s \leq n-1$ and for some $0\leq j \leq n-1$. 
\end{enumerate}

\subsection{Case (I): $\sigma$-conjugation of $V_L(1)$ is of untwisted type}\label{S:1}

By the assumption, $V_{\lambda+L}(r)$ is a simple current module of $V_L^{\hat{g}}$ and hence $(1-g)\lambda \in L$ (see for example \cite{ALY} or \cite{Lam19}).
Since $g$ is completely fixed point free of order $n$, the minimal polynomial of $g$ on $L$ is the $n$-th cyclotomic polynomial $\Phi_n(x)$
and the characteristic polynomial of $g$ on $L$ is  $\Phi_n(x)^{\ell/\varphi(n)}$, where $\ell=\mathrm{rank}(L)$  
and $\varphi$ is the Euler totient function. 
Therefore,
\[  
\dim V_L(j)_1= 
\begin{cases}
\frac{\ell}{\varphi(n)},  & \text{ if }(j,n)=1,\\
0,& \text{ otherwise.} 
\end{cases}
\]
Hence, we have $\dim V_{\lambda+L}(r)_1= \frac{\ell}{\varphi(n)}$, also. 
Since $g$ stabilizes $\lambda+L$, it induces an isometry on $\lambda+L$. Since $g$ is completely fixed point free,  $g^i\alpha\neq \alpha$ for any $\alpha\in \lambda+L$ and $1\leq i\leq n-1$ and thus we have $\dim V_{\lambda+L}(r)_1=|(\lambda+L)(2)|/n$ for any $0\leq r\leq n-1$. Therefore,  
\begin{equation}
|(\lambda+L)(2)|= \frac{n}{\varphi(n)}\cdot\mathrm{rank}\,L.\label{Eq:lambda2}
\end{equation}
In particular,  $ (\lambda+L)(2) \neq \emptyset$.

Since $\Phi_n(g) \lambda=0$ and $g$ stabilizes $\lambda+L$, we have $\Phi_n(1)\lambda \in L$. 
By Lemma \ref{phi(1)},  $\Phi_n(1)=1$ if $n$ is not a prime power and $\Phi_n(1)=p$ if $n=p^t$ is a prime power. 

Now set $N={\rm Span}_\Z\{L, \lambda\}$.  Then we have $|N/L|=1$ if $n$ is not a prime power and $|N/L|=p$ if $n=p^t$ is not a prime power. 

By our assumption, $L(2)=\emptyset$  but $N(2) \neq \emptyset$. Therefore, $|N/L|>1$ and $n$ must be a prime power, say $p^t$ and  $|N/L|= p$.   Since $(1-g)\lambda\in L$, $\hat{g}$ stabilizes  $V_{j\lambda+L}$ for any $0\leq j\leq p-1$ and hence $\hat{g}$ also acts on $V_N$. In particular, for each $j$,
\[
V_{j\lambda+L} = \oplus_{s=0}^{n-1}  V_{j\lambda+L}(s), 
\] 
where $V_{j\lambda+L}(s)= \{ v\in V_{j\lambda+L}\mid  \hat{g} v=  e^{2s\pi\sqrt{-1}/ n} v \} $.
 
 \begin{lemma} \label{fusion1}
 For  $u\in V_{j\lambda+L}(s)$, $v\in V_{j'\lambda+L}(s')$ and $n\in \Z$, we have 
 \[
 u_nv \in  V_{(j+j')\lambda +L} (s+s'). 
  \]
 In particular, we have  $V_{j\lambda+L}(s) \boxtimes_{V_L^{\hat{g}}} V_{j'\lambda+L}(s') \cong  V_{(j+j')\lambda +L} (s+s')$. 
 \end{lemma}
 
Now assume that $n=p^t$ is a prime power and set $m=n/p= p^{t-1}$.  Let $h=g^m$. Then $h$ is fixed point free of order $p$ on $L$. Moreover, we have 

(1)  $|N(2)| = \sum_{i=1}^{p-1} |(i\lambda +L)(2)| =(p-1) \frac{p^t\ell}{p^{t-1}(p-1)}=p\ell$. 

(2)  $h(\lambda+L) =\lambda+L$.

By \cite[Proposition 1.8]{Sh04} and \cite[Theorem 4.19]{LS21},  we have the following result.

\begin{lemma}\label{Lem:Ap}
The sublattice of $N$ spanned by $N(2)$ is isometric to the orthogonal sum of $k$ copies of $A_{p-1}$, where $k=\ell/(p-1)$. Therefore, $N$ can be obtained by construction A from a certain code $C$ over $\Z_p$ and  $L$ can be obtained by construction B from the same code $C$.  
\end{lemma}

Let $R=\oplus_{i=1}^k R_i \cong A_{p-1}^k$ be the root lattice spanned by $N(2)$, where $R_i\cong A_{p-1}$ is a simple root lattice of type $A_{p-1}$. Let $\Delta_i=\{\alpha_1^i, \dots,\alpha_{p-1}^i \}, i=1, \dots, k,$ be a set of simple roots (or a base) of $R_i$ and denote $\Delta=\cup_{i=1}^k {\Delta_i}$ . For each $i$, let $g_{\Delta_i}$ be 
a cyclic permutation of order $p$ on $\tilde{\Delta}_i$
such that 
\[
\alpha_1^i \mapsto \alpha_{2}^i \mapsto \dots \mapsto \alpha_{p-1}^i \mapsto \alpha_0^i \mapsto  \alpha_1^i, 
\]
where $\tilde{\Delta}=\Delta\cup\{\alpha_0^i\}$ and $\alpha_0^i$ is the negated highest root.

Note that $A_{p-1}^*/A_{p-1}\cong \Z_p$ and $N/R< \Z_p^k$. For $e=(e_1, \dots, e_k)\in\bigoplus_{i=1}^t\Z_{p}^*$, set 
\begin{equation}
	g_{\Delta,e}=((g_{\Delta_1})^{e_1},\dots,(g_{\Delta_k})^{e_k})\in O(N). \label{Eq:gde}
\end{equation}
Then  $g_{\Delta,e}$ is fixed point free of order $p$ on $N$. 

By \cite[Theorem 4.19]{LS21} again,  $h$ corresponds to $g_{\Delta,e}$ with respect to some $e\in C^\perp$ of Hamming weight $k$. 


By \cite[Theorem 5.3]{LS21}, there is  a standard lift $\hat{h}$ of $h$ and an automorphism $\sigma\in V_L^{\hat{h}}$ such that  $V_L\circ \sigma \cong V_N^{\hat{h}}$ as  $V_L^{\hat{h}}$-modules. 
By adjusting the lift $\hat{g}$ of $g$, we may also assume $\hat{h}=\hat{g}^m$, where $m=n/p$ (cf. \cite[Lemma 4.5]{LS20}). 

\medskip

In this case, we have  
\[
V_L(1; \hat{h})\circ \sigma \cong V_{\lambda+L}^{\hat{h}}, 
\] 
where $V_L(j; \hat{h})=\{v\in V_L\mid\hat{h}v= e^{2\pi \sqrt{-1}\frac{j}{p} } v\}$. 
Note also  that $V_{\lambda+L}^{\hat{h}}= \bigoplus_{i=1}^{m-1} V_{\lambda+L}(ip; \hat{g})$. 

Since  $V_L(1) \circ \sigma \cong  V_{\lambda +L}(r)$ and $n$ is the smallest integer such that $V_L(1)^{\boxtimes n} \cong V_L(0)$, we have $ V_{\lambda+L}(r)^{\boxtimes s} \cong V_{s\lambda+L}(sr)\ncong V_L(0)$  if  $s <n$.   
Moreover,  $ V_{\lambda+L}(r)^{\boxtimes s} \cong V_{s\lambda+L}(sr)\cong V_L(0)$ if and only if  $p|s$ and $sr\equiv 0\mod n$.      
 Therefore,  $(m,r)=1$.  On the other hand, we have  $V_{\lambda+L}(r)^{\boxtimes (1+ip)} \cong V_{\lambda+L}(r+ irp)$ and  thus 
\[
V_L(1; \hat{h})\circ \sigma =\bigoplus_{i=1}^{m-1} V_{L}(1+ip; \hat{g})\circ \sigma= 
\bigoplus_{i=1}^{m-1} V_{\lambda+L}(r+ip; \hat{g}).
\]
Therefore, we have $r\equiv 0\mod p$ and thus $(p,m)=1$; nevertheless, $n=p^t$ is a prime power and thus $m=n/p=1$ and $n=p$ is a prime number.  Recall that the case when $n=p$ is a prime has been studied in details in \cite{LS21}.   It turns out that $L$ is always obtained by Construction B  from a code over $\Z_p$ and $g$ corresponds to a Coxeter element of a root system of type $A_{p-1}^{\ell/(p-1)}$ (cf. \cite[Main Theorem]{LS21}). 

\medskip 

\subsection{Case (II): $\sigma$-conjugation of $V_L(1)$ is of twisted type}\label{S:twistedtype}

In this case,  $V_L(1)\circ \sigma \cong V_{\mu+L}^T[\hat{g}^s](j)$ for some $ 1\leq s\leq n-1$ and $0\leq  j \leq n-1$.

Since $g^s$ is still  fixed point free on $L$ for any $ 1\leq s\leq n-1$, the irreducible $\hat{g}^s$-twisted module $V_{\mu+L}^T[\hat{g}^s]$ is given by 
\begin{equation*}
	V_{\lambda+L}^T[\hat{g}^s]= M(1)[{g}^s]\otimes T_{\tilde{\mu}}. 
\end{equation*}

Let $d=(n,s)$ and $m=n/d$. Then $\hat{g}^s$ has order $m$ and the conformal weight of $V_{\lambda+L}^T[\hat{g}^s]$ (see \cite{Le,DL}) is given by 
\begin{equation}
	\varepsilon(s)=\frac{1}{4m^2} \sum_{i=1}^{m-1} i(m-i) \dim \mathfrak{h}_{(i; g^s)}.\label{Eq:esp}
\end{equation}
Since $g$ is completely fixed point free, $\mathfrak{h}_{(i; g^s)} =0$ unless  $(m,i)=1$. Moreover,  $ \dim \mathfrak{h}_{(i; g^s)}= \dim \mathfrak{h}_{(j; g^s)}$ if $(m,i)=(m,j)=1$. 
Thus, 
\[
	\begin{split}
		\varepsilon(s)& = \frac{1}{4m^2} 	\left(\sum_{1\leq i \leq m-1,\atop (m,i)=1} i(m-i) \right) \cdot \frac{\ell}{\varphi(m)}, \\
	\end{split}
\]

We now compute the values of $\displaystyle \sum_{1\leq i \leq m-1,\atop (m,i)=1} i(m-i)$. 

\begin{lemma}\label{sumi(n-i)}
	Let  $n=\prod_{i=1}^k p_i^{r_i}$ be the prime factorization of $n$, where $p_i, 1\leq i\leq k,$ are prime numbers. Let $I\subset \{p_1, p_2, \dots, p_k\}$ be a subset of primes.  
	Then 
	\[
	\sum_{1\leq i\leq n-1,\atop (p_j,i)=1, p_j\in I} i(n-1)= \frac{1}{6} n \prod_{p_j\in I}(1-\frac{1}{p_j})\left [n^2+(-1)^{|I|+1} \prod_{p_j\in I} p_j \right].
	\]
\end{lemma}

\begin{proof}
	We will prove the lemma by induction on $|I|$.
	
	First we recall (see for example \cite{LS21}) that 
	\[
	\sum_{i=1}^{n-1} i(n-i) = \frac{1}6 n(n-1)(n+1). 
	\] 
	
	If $n=p$ is a prime, then $(n,i)=1$ for any $1\leq i\leq n-1$ and 
	\[
	\sum_{1\leq i\leq p-1,\atop (p,i)=1} i(n-i) = \frac{1}6 p(p-1)(p+1)
	\]
	as desired.  We may assume that $n$ is not a prime.

	Suppose $I=\{p\}\subset \{p_1, p_2, \dots, p_k\}$, i.e., $|I|=1$ and assume that $n=p^rm$, where $(m,p)=1$.  Then  
	\[
	\begin{split}
		\sum_{1\leq i\leq n-1,\atop (p,i)=1} i(n-i)= &\sum_{1\leq i\leq n-1, 
		} i(n-i) -  \sum_{1\leq i\leq n-1,\,  p| i} i(n-i),\\
		=&  \frac{1}6 n(n-1)(n+1) -  p^2\cdot \frac{1}6 p^{r-1}m(p^{r-1}m-1)(p^{r-1}m+1), \\
		= &  \frac{1}6 n(p-1)(p^{2r-1}m^2+1)= \frac{1}6 n(1-\frac{1}p)(n^2+p). 
	\end{split}
	\]
Therefore, the lemma holds for $|I|=1$. 
	
Now assume that the lemma holds for any $I$ with $|I|<k$ and let $p\in \{p_1, p_2, \dots, p_k\}\setminus I$. Set $J=I\cup \{p\}$ and $n=p^r m$ with $(p,m)=1$. 
Then 
\[
\begin{split}
	\sum_{1\leq i\leq n-1,\atop {(p_j,i)=1, p_j\in J}} i(n-i)& = \sum_{1\leq i\leq n-1, 
		\atop (p_j,i)=1, p_j\in I} i(n-i) -  \sum_{1\leq i\leq n-1,\,  p| i
	\atop (p_j,i)=1, p_j\in I} i(n-i),\\
	&=  \frac{1}{6} n \prod_{p_j\in I}(1-\frac{1}{p_j})\left [n^2+(-1)^{|I|+1} \prod_{p_j\in I} p_j \right]\\
	 &\ \  - p^2\cdot \frac{1}{6}  p^{r-1}m \prod_{p_j\in I}(1-\frac{1}{p_j})\left [(p^{r-1}m)^2+(-1)^{|I|+1} \prod_{p_j\in I} p_j \right],\\
	&=  \frac{1}{6} n \prod_{p_j\in I}(1-\frac{1}{p_j}) \left[(p-1)(p^{2r-1}m^2+(-1)^{|I|+2} \prod_{i=1}^k p_i)\right], \\
	&=  \frac{1}{6} n \prod_{p_j\in I}(1-\frac{1}{p_j}) (1-\frac{1}{p})\left[n^2+(-1)^{|I|+2} p\prod_{i=1}^k p_i\right], \\
\end{split}
\]
as desired.	
\end{proof}

\begin{corollary}\label{sum}
	Let $n=\prod_{i=1}^{k} p_i^{r_i}$ be the prime factorization of $n$. Then we have 
	\[
	\sum_{1\leq i \leq n-1,\atop (n,i)=1} i(n-i)= \frac{1}{6} n \prod_{i=1}^k (1-\frac{1}{p_i})\left [n^2+(-1)^{k+1} \prod_{i=1}^k p_i \right] =\frac{\varphi(n)}{6}\left [n^2+(-1)^{k+1} \prod_{i=1}^k p_i \right].
	\]
\end{corollary}


By Corollary \ref{sum}, we have
\begin{equation}\label{phis}
	\begin{split}
		\varepsilon(s)& = \frac{1}{4m^2} 	\left(\sum_{1\leq i \leq m-1,\atop (m,i)=1} i(m-i) \right) \cdot \frac{\ell}{\varphi(m)}, \\
		&= \frac{\ell}{24}\left[1+ \frac{(-1)^{t+1} \prod_{i=1}^{t} q_i}{m^2}\right],   
	\end{split}
\end{equation}
where $m=n/(n,s)$ and $m=\prod_{i=1}^t q_i^{a_i}$ is the prime factorization of $m$.

That $V_L(1)\circ \sigma\cong V_{\mu+L}^T[\hat{g}^s](j)$  implies 
$\varepsilon(s) \leq 1$ and $\varepsilon(s) \in \frac{1}m \Z$. In fact, $\varepsilon(s)= 1-1/m$ or $1$ because $\dim(V_{\mu+L}^T[\hat{g}^s](j))_1= \dim \mathfrak{h}_{(1; g)}$. 
In this case, $V_{\mu+L}^T[\hat{g}^s](j)$ is a simple current module for $V_L^{\hat{g}}$; hence we have $(1-g^s)L^*\leq L$ and $\dim T_{\tilde{\mu}} =[L: (1-g^s)L^*]^{1/2}$ (cf. \cite[Corollary 3.7]{ALY}).  
\medskip

\textbf{Case 1:  $\varepsilon(s)=1- 1/m$.}
Since  $\dim V_{\mu+L}^T[\hat{g}^s](j)_1= \dim V_L(1)_1= {\ell}/\varphi(n)$,  we have $\dim T_{\tilde{\mu}}=1$ and $L=(1-g^s)L^*$. Moreover, 
\[
|L^*/L|= |L/(1-g^s)L|=|\det(1-g^s)|.  
\]
Recall that the characteristics polynomial $\chi_{g^s}(x)=\det(g^s-xI)$ of $g^s$ is given by $\chi_{g^s}(x)=\Phi_m(x)^{\frac{\ell}{\varphi(m)}}$. 
By Lemma \ref{phi(1)},   
\[
|L^*/L|=
\begin{cases}
	1 & \text{ if } m \text{ is not a prime power},\\
	p^{\frac{\ell}{p^{k-1}(p-1)}} &\text{ if } m=p^k \text{ is a prime power}.
\end{cases}
\]

First, we suppose $m$ is a  not a prime power. Then $L$ is even unimodular and $L_2=\emptyset$. 
Moreover,  $$1 > \varepsilon(s) = \frac{\ell}{24}\left[1+ \frac{(-1)^{t+1} \prod_{i=1}^{t} q_i}{m^2}\right] \geq  \frac{\ell}{24}\left(1 -\frac{1}{m}\right).$$ 
Since $m$ is not a prime power,
$m\geq 6$ and we have $1 > \varepsilon(s) \geq \frac{\ell}{24}\left(1 -\frac{1}{6}\right)$; hence, $\ell < 29$. Therefore,   $L$ is isometric to the Leech lattice $\Lambda$; the unique even unimodular lattice of rank $<32$ and with no roots.    

In this case, $\ell=24$ and 
\[
\varepsilon(s) =1 + \frac{(-1)^{t+1}\prod_{i=1}^{t} q_i}{m^2} =1-\frac{1}m. 
\]
It implies $m=\prod_{i=1}^{t} q_i$ and $t$ is even. In addition, $\varphi(m)| 24$. By direct calculations, it is straightforward to show that $m=6,10,14,26,15,21$, or $39$.  The corresponding fixed point free isometry of the Leech lattice belongs to the conjugacy class $-3A, -5A$, $-7A$, $-13A$, $15A$,  $21A$, or $39A$.

\begin{theorem}\label{isoLeech}
	Let $g\in O(\Lambda)$ be of the conjugacy $-3A, -5A,-7A, -13A, 15A, 21A,$ or $39A$ and let $\hat{g}$ be a lift of $g$ in $\Aut(V_\Lambda)$. Then the VOA $V_\Lambda^{orb(g)}$ obtained by a orbifold construction from $V_\Lambda$ and $g$ is isomorphic to $V_\Lambda$.  
\end{theorem}

\begin{proof}
By a direct calculation, it is easy to show that the conformal weight of the twisted module $V_{\Lambda}^T(\hat{g}^{s})$ is $1 -1/|g|$  for any $s$ with $(s, |g|)=1$ and is $>1$ if $(s,|g|)\neq 1$. For $s$ with $(s, |g|)=1$, $\dim V_{\Lambda}^T(\hat{g}^{s})_1= \dim \mathfrak{h}_{(s; g)}$. 
Since $g$ is completely fixed point free,  
$\dim(V_\Lambda^{orb(g)})_1 =\sum_{(s,|g|)=1} \dim \mathfrak{h}_{(s; g)} =24$ and hence $V_\Lambda^{orb(g)}\cong V_\Lambda$. 
\end{proof}

\begin{corollary}
		Let $g\in O(\Lambda)$ be of the conjugacy class $-3A, -5A,-7A, -13A, 15A, 21A,$ or $39A$ and let $\hat{g}$ be a lift of $g$ in $\Aut(V_\Lambda)$.
		Then $V_\Lambda^{\hat{g}}$ has an extra automorphism. 
\end{corollary}
\medskip

Now suppose $m=p^r$ is a prime power. Then 
\[
\ell =\frac{24 (p^r-1) p^r}{p^{2r}+p} =\frac{24 (p^r-1) p^{r-1}}{p^{2r-1}+1}, 
\]
which is an integer and is strictly less than $24$. Notice that  $(p^{2r-1}+1, p^{r-1})=1$ and  
$(p^{2k-1}+1)\gneq (p^k-1)$. Therefore, $(p^{2r-1}+1) |  24 (p^r-1)$.

 If $p=2$, then both $(p^{2r-1}+1)$ and $(p^r-1)$ are odd. Then we have $(2^{2r-1}+1)=3(2^r-1)$ and we have $r=1$ or $2$, i.e., $m=2$ or $4$.  The case $m=2$  has been discussed in \cite{Sh04}  while the case $m=4$ has been studied in \cite{CL21}.  It turns out that the case $m=4$ is not possible, either. 
 
Now assume that $p$ is odd.  Since $\ell \geq \varphi(m)=p^{r-1}(p-1)$, we have $p^{r-1}(p-1)<24$.
Therefore, $m<24$ is a prime or $m=3^2, 3^3$ or $5^2$. 
By direct calculations, it is easy to verify that $(p^{2r-1}+1)$ does not divide $24 (p^r-1)$ if $m=p^r=3^2$, $3^3$ and $5^2$.  
Hence, $m<24$ is a prime. These cases have been studied in \cite{LS21}. It turns out that $L$ is isometric to a coinvariant lattice $\Lambda_g$ of the Leech lattice associated with an element of class $2A,3B,5B,7B,11A,$ or $23A$ ($23B$).  
 
\textbf{Case 2:  $\varepsilon(s)=1$.}  In this case, 
\[
\dim (V_{\mu+L}^T[\hat{g}^s](j))_1 = 
\frac{\dim T_{\tilde{\mu}}}{d} = \frac{\ell}{\varphi(n)}.
\]
By \eqref{phis}, we also have 
\[
\ell = 24 \frac{m^2}{m^2+ (-1)^{t+1} \prod q_i}= \frac{24A}{A+ (-1)^{t+1}},  
\]
where $A=\frac{m^2}{\prod q_i}$. Since $(A, A+ (-1)^{t+1})=1$, $A+(-1)^{t+1}| 24$.  
By direct calculations, it is straightforward to verify that $m$ must be a prime and $m=2,3,5,7, 11$ or $23$.
Moreover, $\ell =16$ (resp., $18, 20, 21, 22$, $23$) if  $m= 2$ (resp., $3,5,7, 11,$ or $23$). 
Again these cases have been discussed in \cite{LS21}.  It turns out that only the cases $m=2,3,5$ are possible and $L$ is isometric to a coinvariant lattice of the Leech lattice associated with an element of class $-2A, 3C$, or $5C$. 

\section{Coinvariant lattices of the Leech lattice}

From discussions in the last section, we know that if $g$ is completely fixed point free  and $V_L^{\hat{g}}$ contains an extra automorphism $\sigma$ such that $V_L(1)\circ \sigma$ is isomorphic  to an irreducible module of twisted type, then  $L$ is  isometric to either
\begin{enumerate}
	\item a coinvariant lattice $\Lambda_h$ of the Leech lattice with $h\in 2A, -2A, 3B, 3C,5B, 5C$, $7B$, $11A$ or $23A\,(23B)$; or
	
	\item  $L\cong \Lambda$ and 
	$g^s\in -3A, -5A, -7A, -13A$, $15A$,  $21A$, or $39A$ for some $1\leq s\leq |g|-1$. 
\end{enumerate}
For all these cases, the irreducible $h$ (resp., $g^s$)-twisted modules have conformal weight in $\frac{1}m \Z$, where  $m=|h|$ (resp., $m=|g^s|$). 

\medskip

In this section, we discuss  $\Aut(V_L^g)$ for  some possible $g$.

\subsection{Coinvariant lattices associated with prime order elements}
We first consider the coinvariant lattices  of the Leech lattice associated with prime order elements.

\begin{lemma}\label{p1}
	Let $L=\Lambda_h$ with  $h\in 2A, -2A, 3B, 3C,5B, 5C$, $7B$, $11A$ or $23A\,(23B)$. 
	Let $g\in O(L)$ be a completely fixed point free isometry such that  $g^n=h$ on $L$. 
	Suppose there is a $\sigma\in V_{L}^{\hat{g}}$ such that $V_L(1)\circ \sigma\cong V_{\lambda+L}^T[\hat{g}^n](j)$ for some $0\leq j\leq |g|-1$.  
	Then $\sigma$ stabilizes $V_L^{\hat{h}}= \oplus_{i=0}^{m-1} V_L(pi) $ by the conjugate action, where $p=|h|$. 
\end{lemma}

\begin{proof}
For $L\cong \Lambda_h$,  $\mathcal{D}(L) =L^*/L$ is an elementary $p$-group, where $p=|h|$. 
Since the conformal weights of the irreducible $h$-twisted modules are in $\frac{1}p \Z$, 
$\mathrm{Irr}(V_L^{\hat{h}})$ also forms an an elementary $p$-group with respect to the fusion product by \cite[Theorem 5.3]{Lam19}. We have $(V_{\lambda+L}^T[\hat{g}^n]_\Z)^{ \boxtimes p}  = V_L^{\hat{h}}$ as  $V_L^{\hat{h}}$-modules. Therefore, $V_{\lambda+L}^T[\hat{g}^n](j)^{\boxtimes p} < V_L^{\hat{h}}$ and   $$V_{\lambda+L}^T[\hat{g}^n](j)^{\boxtimes p}\cong V_L(pi)$$  for some $i$ as $V_L(p)$ has the top weight $>1$. In particular, $V_L^{\hat{h}}= \oplus_{i=0}^{m-1} V_L(pi) $ is stabilized by $\sigma$. 
\end{proof}

As a consequence, we have 
\begin{proposition}\label{primecases}
	Let $L=\Lambda_h$ with  $h\in 2A, -2A, 3B, 3C,5B, 5C$, $7B$, $11A$ or $23A\,(23B)$. 
	Let $g\in O(L)$ be a completely fixed point free isometry such that  $g^n=h$ on $L$.
	Then, $$\Aut(V_L^{\hat{g}})\cong N_{\Aut(V_L^{\hat{h}})}(\langle \bar{g} \rangle)/\langle \bar{g} \rangle,$$ 
	where $\bar{g}$ denotes the restriction of $\hat{g}$ on $V_L^{\hat{h}}$. 
\end{proposition}

\begin{proof}
	For any $\sigma \in \Aut(V_L^{\hat{g}})$, either $\sigma$ stabilizes the set 
	$\{V_L(i)|0\leq i\leq |g|\}$ or  $V_L(1)\circ \sigma\cong V_{\lambda+L}^T[\hat{g}^n](j)$ for some $0\leq j\leq |g|-1$. In both cases,  $\sigma$ stabilizes  $V_L^{\hat{h}}= \oplus_{i=0}^{m-1} V_L(pi) $ by Lemma \ref{p1} and we have the desired result. 
\end{proof}

\subsection{Elements of non-prime orders}
Next we consider the case when 
$L\cong \Lambda$ and 
$g^s\in -3A, -5A, -7A, -13A$, $15A$,  $21A$, or $39A$ for some $1\leq s\leq |g|-1$. 
By the same arguments as in Lemma \ref{p1} and Proposition \ref{primecases}, we have the following result.

\begin{proposition}
Let $g\in O(\Lambda)$ such that $h=g^s\in  -3A, -5A, -7A, -13A$, $15A$,  $21A$, or $39A$.
Then  $\Aut(V_\Lambda^{\hat{g}})\cong N_{\Aut(V_\Lambda^{\hat{h}})}(\langle \bar{g}\rangle)/ \langle \bar{g}\rangle$, where $\bar{g}$ denotes the restriction of $\hat{g}$ on $V_{\Lambda}^{\hat{h}}$. 
\end{proposition}

\section{Some explicit examples} 
Next  we consider some explicit examples.

\subsection{$h=g^s\in 2A$ for some positive integer $s$}  
Let $h\in O(\Lambda)$ be of the class $2A$. Then $\Lambda_h\cong \sqrt{2}E_8$, $O(\Lambda_h) \cong \mathrm{Weyl}(E_8)\cong 2. \Omega^+_8(2).2$ and $h$ acts as $-1$ on $\Lambda_h$.  In addition, $\Aut(V_{\Lambda_h}^{\hat{h}}) \cong GO^+_{10}(2)$.

The Weyl group of $E_8$ contains $5$ conjugacy classes of completely fixed point free isometry $g$ such that $-1\in \langle g\rangle$. They correspond to  $-3B, -5B, -15B$, $2C$ (or $2D$) , $4E$ (or $4F$).    

Recall that $\dim(V_{\Lambda_h}^{\hat{h}})_2=156$ and $(V_{\Lambda_h}^{\hat{h}})_2$ decomposed as as two irreducible modules of $GO^+_{10}(2)$ with dimensions $1$ and $155$.  A lift of the corresponding isometry of $\sqrt{2}E_8$ acts on $(V_{\Lambda_h}^{\hat{h}})_2$. By direct calculations and using the character tables, the corresponding traces and conjugacy classes are listed in Table \ref{Table:2A1}.

{\tiny
\begin{longtable}{|c|c|c|c|c|c|}
	\caption{Traces of $\hat{g}$ on $(V_{\Lambda_h}^{\hat{h}})_2$}\label{Table:2A1}
	\\ \hline 
	Classes in $\mathrm{Weyl}(E_8)$ & Trace of $\hat{g}$ on $(V_{\Lambda_h}^{\hat{h}})_2$ & Classes in $GO^+_{10}(2)$ 
	& $|C_{O(\Lambda_h)}(g)|$ & $|C_{\Aut(V_L^{\hat{h}})}(\bar{g} )|$ &	
	$N_{\Aut(V_L^{\hat{h}})}(\langle \bar{g} \rangle)/ \langle \bar{g} \rangle$ 	 \\ \hline    
	$-3B$ & $6$ &  $3C$ & $77760$ &$77760$ & $PSU_4(2).2$ \\ 
	$-5B$ & $1$  & $5B$ & $300$ & $300$ & $ Sym_5$\\
	$-15B$ & $1$ &  $15E$ & $15$ &$15$ & $2$\\ 
	$\widehat{2C}, \widehat{2D}$  (order 4)& $-4$ & $2C$ & $2^{10}3^2 5$ & $2^{18}3^2 5$ & $[2^{14}]. \Sym_{6}$\\
	$\widehat{4E},\widehat{4F}$ (order 8) & $0$ & $4H$  & $2^{6} 3$ & $2^{10} 3$ & $[2^{6}].\Sym_4$ \\ \hline
\end{longtable} 
}

\begin{remark}
For the conjugacy classes $-3B$, $-5B$ and $-15B$, $\Aut(V_{\Lambda_h}^{\hat{g}})$ is isomorphic to $N_{\aut(V_{\Lambda_h})}(\langle \hat{g}\rangle)/ \langle \hat{g}\rangle$; there are no extra automorphisms for these cases. 
 
 For  the conjugacy classes  $\widehat{2C}, \widehat{2D}$ (resp., $\widehat{4E},\widehat{4F}$), $g$ has order $4$ (resp., order $8$).  For these two cases, 
  $\Aut(V_{\Lambda_h}^{\hat{g}})$ is slightly bigger than $N_{\aut(V_{\Lambda_h})}(\langle \hat{g}\rangle)/ \langle \hat{g}\rangle$. 
\end{remark}

\medskip

\subsection{Some fixed point free elements of $\Lambda$ with non-prime orders} 
Next we consider  some fixed point free elements of $\Lambda$ with non-prime orders.
Let  $g\in O(\Lambda)$ such that $g\in -3A, -5A, -7A, -13A$, $15A$,  $21A$, or $39A$.  
For these cases, consider the sets 
\[
\mathcal{V}= \{V_\Lambda(i)\mid  (|g|, i)=1\} \quad \text{ and } \quad
\mathcal{V}_t= \{V_\Lambda[\hat{g}^i](0)\mid  (|g|, i)=1\}.  
\]
Then  for any $\sigma \in \Aut(V_{\Lambda}^{\hat{g}} )$, we have $V_\Lambda(1)\circ \sigma \in 
\mathcal{V} \cup \mathcal{V}_t$.

\begin{lemma}
Let $\sigma \in \Aut(V_{\Lambda}^{\hat{g}} )$ such that $V_\Lambda(1)\circ \sigma \in 
\mathcal{V}_t$.  Then  $V_\Lambda(1)\circ \sigma^2 \in \mathcal{V}$. That means $\sigma^2\in 
\Stab_{\Aut(V_{\Lambda}^{\hat{g}} )}(\mathcal{V}) \cong N_{\aut(V_{\Lambda})}(\langle \hat{g}\rangle)/ \langle \hat{g}\rangle$. 
\end{lemma}

 \begin{proof}
 Let $\sigma \in \Aut(V_{\Lambda}^{\hat{g}} )$ such that $V_\Lambda(1)\circ \sigma \cong 
 V_\Lambda[\hat{g}^i](0)$ for some $i$ with $(i, |g|)=1$.  
 
 Suppose $\sigma^2\notin  \Stab_{\Aut(V_{\Lambda}^{\hat{g}} )}(\mathcal{V})$. Then 
 we have $V_\Lambda(1)\circ \sigma^2 
\cong  V_\Lambda[\hat{g}^j](0)$ for some $j$ with $(j, |g|)=1$. In this case, we have 
\[
\begin{split}
V_\Lambda[\hat{g}^i](1) \circ \sigma   &= (V_\Lambda(1)\times V_\Lambda[\hat{g}^i](0) )\circ \sigma \\ 
& \cong  V_\Lambda[\hat{g}^i](0) \times V_\Lambda[\hat{g}^j](0) \cong   V_\Lambda[\hat{g}^{i+j}](0). 
\end{split}
\]
The irreducible module $V_\Lambda[\hat{g}^{i+j}](0)$ has integral weights but the weights of 
$V_\Lambda[\hat{g}^i](1)$ are not integers, which is a contradiction. Hence, we have  $\sigma^2\in 
Stab_{\Aut(V_{\Lambda}^{\hat{g}} )}(\mathcal{V})$.   	
 \end{proof}

\begin{corollary} \label{index:aut}
Let  $g\in O(\Lambda)$ such that $g\in -3A, -5A, -7A, -13A$, $15A$,  $21A$, or $39A$. Then 
$|\Aut(V_{\Lambda}^{\hat{g}} )/ \Stab_{\Aut(V_{\Lambda}^{\hat{g}} )}(\mathcal{V})|=2$. 
\end{corollary}

The automorphism groups of  $\Aut(V_{\Lambda}^{\hat{g}} )$ for $g\in -3A, -5A, -7A, -13A$, $15A$,  $21A$, or $39A$ are as follows. We refer to \cite{Atlas} and \cite{Cur80} for the shapes of $C_{O(\Lambda)}(g)$ and $N_{O(\Lambda)}(\langle g\rangle)$.

\begin{longtable}{|c|c|c|c|}
	\caption{Automorphism groups of  $\Aut(V_{\Lambda}^{\hat{g}} )$}\label{Table:nonprime}
	\\ \hline 
	$g\in O(\Lambda)$ & $C_{Co_0}(g)$ & $N_{Co_0}(\langle g\rangle)$ & $\Aut(V_\Lambda^{\hat{g}})$	 \\ \hline    
	$-3A$ & $6.\mathrm{Suz}$ &  $6.\mathrm{Suz}{:}2$ 
	& $\mathrm{Suz}{:}2 \times 2$  \\ 
	$-5A$ & $2.(5\times J_2) $  & $2.(5\times J_2){:}2$ & $ J_2{:}2 \times 2$ \\
	$-7A$ & $2.(7 \times Alt_7)$ &  $2.(7{:}3 \times Alt_7){:}2$ & $(3\times Alt_7){:}2 . 2$ \\ 
	$-13A$& $2.(13 \times Alt_4)$ & $2.(13{:}6 \times Alt_4):2$ & $(6 \times Alt_4):2. 2$ \\
	$15A$ & $2.(5\times 3.Alt_6)$& $2.(5\times 3.Alt_6.2).2$ &  $ 2. Alt_6.2^2 \times 2$\\
	$21A$  & $2.(Alt_4 \times 21)$ & $2.( (Alt_4\times 3){:} 2 \times 7{:} 3): 2$   &  $2.( Alt_4{:} 2 \times 3): 2.2$    \\ 
	   $39A$& $2.39$ & $2.(3\times 13:6):2$  & $2.6.2.2$\\ \hline 
\end{longtable} 

\end{document}